\documentclass[12pt]{amsart}
\usepackage{amsmath,amsthm,amsfonts,amssymb,times}

\usepackage{titlesec}
\titleformat{\section}  
{\normalfont\Large\bfseries}
{\thesection}{1em}{}
\titleformat{\subsection}
{\normalfont\large\bfseries}
{\thesubsection}{1em}{}

\numberwithin{equation}{section}  

\DeclareMathAlphabet{\curly}{U}{rsfs}{m}{n}  

\theoremstyle{remark}

\theoremstyle{plain}

\newtheorem{lem}{Lemma}[section]
\newtheorem{thm}{Theorem}
\newtheorem{cor}{Corollary}

\numberwithin{equation}{section}

\newcommand{\ZZ}{{\mathbb Z}}

\DeclareMathOperator{\lcm}{lcm}
\DeclareMathOperator{\li}{li}

\makeatletter
\renewcommand{\pmod}[1]{\allowbreak\mkern7mu({\operator@font mod}\,\,#1)}
\makeatother

\newcommand{\bal}{\[\begin{aligned}}
\newcommand{\eal}{\end{aligned}\]}

\newcommand{\be}{\begin{equation}}
\newcommand{\ee}{\end{equation}}

\newcommand{\ssum}[1]{\sum_{\substack{#1}}}  
\newcommand{\sprod}[1]{\prod_{\substack{#1}}} 

\renewcommand{\a}{\ensuremath{\alpha}}

\renewcommand{\le}{\leqslant}

\renewcommand{\ge}{\geqslant}

\newcommand{\order}{\asymp}      
\renewcommand{\(}{\left(}
\renewcommand{\)}{\right)}

\newcommand{\pfrac}[2]{\left(\frac{#1}{#2}\right)}  

\newcommand{\bb}{\ensuremath{\mathbf{b}}}
\newcommand{\vv}{\ensuremath{\mathbf{v}}}

\newcommand{\flr}[1]{\left\lfloor #1 \right\rfloor}

\newcommand{\er}{{\rm e}}

\begin{document}

\title{The image of Carmichael's $\lambda$-function}
\author{Kevin Ford, Florian Luca, and Carl Pomerance}
\date{\today}
\address{KF: Department of Mathematics, 1409 West Green Street, University
of Illinois at Urbana-Champaign, Urbana, IL 61801, USA}
\email{ford@math.uiuc.edu}
\address{FL:  Instituto de Mat\'ematicas, UNAM Juriquilla,  Santiago de Quer\'etaro, 76230  Que-r\'etaro de Arteaga, M\'exico and School of Mathematics, University of the Witwatersrand, P. O. Box Wits 2050, South Africa}
\email{fluca@matmor.unam.mx}
\address{
CP: Mathematics Department, 
Dartmouth College,
Hanover, NH 03755, USA}
\email{carl.pomerance@dartmouth.edu}

\thanks{KF was supported in part by National Science Foundation grant
DMS-1201442.  CP was supported in part by NSF grant DMS-1001180.  
Part of this work was done while KF and FL visited Dartmouth
College in Spring, 2013.  They thank the people there for their
hospitality.  
CP gratefully acknowledges
a helpful conversation with Andrew Granville in which the
heuristic argument behind our proof first arose.
The authors also thank
one of the referees for constructive comments which improved the paper.
}
\begin{abstract}
In this paper, we show that the counting function of the set of values of the Carmichael $\lambda$-function is $x/(\log x)^{\eta+o(1)}$, where $\eta=1-(1+\log\log 2)/(\log 2)=0.08607\ldots$.
\end{abstract}

\maketitle


\section{Introduction}
Euler's function $\varphi$ assigns to a natural number $n$ the order of the group of units of
the ring of integers modulo $n$.  It is of course ubiquitous in number theory, as is its close
cousin $\lambda$, which gives the exponent of the same group.  Already appearing in
Gauss's {\it Disquisitiones Arithmeticae}, $\lambda$ is commonly referred to as
Carmichael's function after R. D. Carmichael, who studied it about a century ago.
(A {\it Carmichael number} $n$ is composite but nevertheless satisfies $a^n\equiv a\pmod n$
for all integers $a$, just as primes do.  Carmichael discovered these numbers which
are characterized by the property that $\lambda(n)\mid n-1$.)

It is interesting to study $\varphi$ and $\lambda$ as functions.  For example, how
easy is it to compute $\varphi(n)$ or $\lambda(n)$ given $n$?  It is indeed easy if we
know the prime factorization of $n$.  Interestingly, we know the converse.
After work of Miller \cite{M}, 
given either $\varphi(n)$ or $\lambda(n)$, it is easy to find the prime factorization of $n$.

Within the realm of ``arithmetic statistics'' one can also ask for the behavior of
$\varphi$ and $\lambda$ on typical inputs $n$, and ask how far this varies from
their values on average.  For $\varphi$, this type of question goes back to the
dawn of the field of probabilistic number theory with the seminal paper of
Schoenberg \cite{S}, while some results in this vein for $\lambda$ are found in 
\cite{EPS}.

One can also ask about the value sets of $\varphi$ and $\lambda$.  That is, what
can one say about the integers which appear as the order or exponent of the groups
$(\ZZ/n\ZZ)^*$?

These are not new questions.
Let $V_\varphi(x)$
denote the number of positive integers $n\le x$ for which $n=\varphi(m)$ for some $m$.  Pillai \cite{Pi}
showed in 1929 that $V_\varphi(x)\le x/(\log x)^{c+o(1)}$ as $x\to\infty$, where $c=(\log 2)/\er$.
On the other hand, since $\varphi(p)=p-1$, $V_\varphi(x)$ is at least $\pi(x+1)$, 
the number of primes in $[1,x+1]$, and so $V_\varphi(x)\ge(1+o(1))x/\log x$.  In one of his earliest
papers, Erd\H os \cite{E} showed that the lower bound is closer to the truth: we have
$V_\varphi(x)=x/(\log x)^{1+o(1)}$ as $x\to\infty$.  This result has since been refined by
a number of authors, including Erd\H os and Hall, Maier and Pomerance, and Ford, see \cite{F}
for the current state of the art.

Essentially the same results hold for the sum-of-divisors function $\sigma$, but only recently
\cite{FLP} were we able to show that there are infinitely many numbers that are simultaneously
values of $\varphi$ and of $\sigma$, thus settling an old problem of Erd\H os.  

In this paper, we address the range problem for Carmichael's 
function $\lambda$.  
From the definition of $\lambda(n)$ as the exponent of the group $(\ZZ/n\ZZ)^*$,
it is immediate that $\lambda(n)\mid\varphi(n)$ and that $\lambda(n)$ is
divisible by the same primes as $\varphi(n)$.  In addition, we have
$$
\lambda(n)={\text{\rm lcm}}[\lambda(p^a): p^a\,\|\, n],
$$
where
$\lambda(p^a)=p^{a-1}(p-1)$ for odd primes $p$ with $a\ge 1$
or $p=2$ and $a\in \{1,2\}$.  Further, $\lambda(2^a)=2^{a-2}$ for $a\ge 3$. 
Put $V_{\lambda}(x)$ for the number of integers $n\le x$ with $n=\lambda(m)$ for some $m$.
Note 
that since $p-1=\lambda(p)$ for all primes $p$, it follows that 
\begin{equation}
\label{eq:low}
V_{\lambda}(x)\ge \pi(x+1)=(1+o(1))\frac{x}{\log x}\quad (x\to\infty),
\end{equation}
as with $\varphi$.  In fact, one might suspect that the story for $\lambda$ is completely
analogous to that of $\varphi$.  As it turns out, this is not the case.

It is fairly easy to see that $V_\varphi(x)=o(x)$ as $x\to\infty$,
since most numbers $n$ are divisible by many different primes, so
most values of $\varphi(n)$ are divisible by a high power of 2.
This argument fails for $\lambda$ and in fact it is not immediately
obvious that $V_\lambda(x)=o(x)$ as $x\to\infty$.  Such a result 
was first shown in \cite{EPS}, where it was established that there
is a positive constant $c$ with $V_\lambda(x)\ll x/(\log x)^c$.
In \cite{FL}, a value of $c$ in this result was computed.
It was shown there that, as $x\to\infty$,
\begin{equation}
\label{eq:up}
V_{\lambda}(x)\le \frac{x}{(\log x)^{\alpha+o(1)}}\quad {\text{\rm holds with}}\quad  \alpha=1-\er(\log 2)/2=0.057913\ldots\,.
\end{equation}
The exponents on the logarithms in the lower and upper bounds \eqref{eq:low} and \eqref{eq:up} were brought closer in 
the recent paper \cite{LP}, where it was shown that, as $x\to\infty$,
$$
\frac{x}{(\log x)^{0.359052}}<V_{\lambda}(x)\le \frac{x}{(\log x)^{\eta+o(1)}}\quad {\text{\rm with}}\quad \eta=1-\frac{1+\log\log 2}{\log 2}=0.08607\ldots.
$$
In Section 2.1 of that paper, a heuristic was presented suggesting that the 
correct exponent of the logarithm should be the number $\eta$.
In the present paper, we confirm the heuristic from \cite{LP} by proving the following theorem. 

\begin{thm}\label{mainthm}
We have $V_\lambda(x)=x(\log x)^{-\eta+o(1)}$, as $x\to\infty$.
\end{thm}

Just as results on $V_\varphi(x)$ can be generalized to similar
multiplicative functions, such as $\sigma$, we would expect our
result to be generalizable to functions similar to $\lambda$
enjoying the property $f(mn)=\lcm[f(m),f(n)]$ when $m,n$ are coprime.

Since the upper bound in Theorem \ref{mainthm}
was proved in \cite{LP}, we need only show
that $V_{\lambda}(x)\ge x/(\log x)^{\eta+o(1)}$ as $x\to\infty$.
We remark that in our lower bound argument we will count only
squarefree values of $\lambda$.

The same number $\eta$ in Theorem \ref{mainthm}
appears in an unrelated problem.  As shown by Erd\H os
\cite{Emult}, the number of distinct entries in the multiplication
table for the numbers up to $n$ is $n^2/(\log n)^{\eta+o(1)}$ as $n\to\infty$.  Similarly,
the asymptotic density of the integers with a divisor in $[n,2n]$ is
$1/(\log n)^{\eta+o(1)}$ as $n\to\infty$.  See \cite{Fdiv} and \cite{Fmult} for more on these kinds of
results.  As explained in the heuristic argument presented in \cite{LP},
the source of $\eta$ in the $\lambda$-range problem comes from the
distribution of integers $n$ with about $(1/\log 2)\log\log n$
prime divisors: the number of these numbers $n\in[2,x]$ is
$x/(\log x)^{\eta+o(1)}$ as $x\to\infty$.  Curiously, the number $\eta$
arises in the same way in the multiplication table problem: most entries
in an $n$ by $n$ multiplication table have about  $(1/\log 2)\log\log n$
prime divisors (a heuristic for this is given in the introduction of \cite{Fdiv}).

We mention two related unsolved problems.  Several papers (\cite{BFPS, BL, Fr, PP}) have discussed
the distribution of numbers $n$ such that $n^2$ is a value of $\varphi$;
in the recent paper \cite{PP} it was shown that the number of such
$n\le x$ is between $x/(\log x)^{c_1}$ and $x/(\log x)^{c_2}$, where
$c_1>c_2>0$ are explicit constants.  Is
the count of the shape $x/(\log x)^{c+o(1)}$ for some number $c$?  The
numbers $c_1,c_2$ in \cite{PP} are not especially close.
The analogous problem for $\lambda$ is wide open.  
In fact, it seems
that a reasonable conjecture (from \cite{PP}) is that 
asymptotically all even
numbers $n$ have $n^2$ in the range of $\lambda$.
On the other hand, it has not been proved that there is a lower
bound of the shape $x/(\log x)^c$ with some positive constant $c$ for the number of such numbers $n\le x$.

%
%
\section{Lemmas}
%
%

Here we present some estimates that will be useful in our argument.
To fix notation, for a positive integer $q$ and an integer $a$,
we let $\pi(x;q,a)$ be the number of primes $p\le x$ in the progression $p\equiv a\pmod{q}$, and put
\[
 E^*(x;q) = \max_{y\le x} \Big| \pi(y;q,1)-\frac{\li(y)}{\varphi(q)} \Big|,
\]
where $\li(y)=\int_2^y {\rm d}t/\log t$.  

We also let $P^+(n)$ and $P^-(n)$ denote the largest prime factor of $n$ and the smallest prime factor of
$n$, respectively, with the convention that $P^-(1)=\infty$ and $P^+(1)=0$.
Let $\omega(m)$ be the number of distinct prime factors of $m$, and let $\tau_k(n)$ be the $k$-th divisor function; that is,
the number of ways to write $n=d_1\cdots d_k$ with $d_1,\ldots,d_k$ positive integers.
Let $\mu$ denote the M\"obius function.

First we present an estimate for the sum of reciprocals of integers with a given number of prime factors.

\begin{lem}\label{omega}
 Suppose $x$ is large. Uniformly for $1\le h\le 2\log\log x$,
 \[
  \ssum{P^+(b)\le x \\ \omega(b)=h} \frac{\mu^2(b)}{b} \order \frac{(\log\log x)^h}{h!}.
 \]
\end{lem}

\begin{proof}
 The upper bound follows very easily from
 \[
   \ssum{P^+(b)\le x \\ \omega(b)=h} \frac{\mu^2(b)}{b} \le \frac{1}{h!} \Big( \sum_{p\le x} \frac{1}{p} \Big)^h
   = \frac{(\log\log x + O(1))^h}{h!} \order \frac{(\log\log x)^h}{h!}
 \]
 upon using Mertens' theorem and the given upper bound on $h$.
 For the lower bound we have
 \[
   \ssum{P^+(b)\le x \\ \omega(b)=h } \frac{\mu^2(b)}{b} \ge \frac{1}{h!} \Bigg( \sum_{p\le x} \frac{1}{p} \Bigg)^h
   \Bigg[ 1 - \binom{h}{2} \Bigg( \sum_{p\le x} \frac{1}{p} \Bigg)^{-2} \sum_{p} \frac{1}{p^2} \Bigg].
 \]
 Again, the sums of $1/p$ are each $\log\log x + O(1)$.
The sum of $1/p^2$ is smaller than $0.46$, hence for large enough $x$ the bracketed expression
is at least $0.08$, and the desired lower bound follows.
\end{proof}

Next, we recall (see e.g., \cite[Ch. 28]{Da}) the well-known theorem of Bombieri and Vinogradov,
and then we prove a useful corollary.

\begin{lem}\label{BV}
 For any number $A>0$ there is a number $B>0$ so that for $x\ge2$,
 \[
  \sum_{q\le \sqrt{x} (\log x)^{-B}} E^*(x;q) \ll_A \frac{x}{(\log x)^A}.
 \]
\end{lem}

\begin{cor}\label{BVcor}
 For any integer $k\ge 1$ and number $A>0$ we have for all $x\ge2$,
 \[
  \sum_{q\le x^{1/3}} \tau_k(q) E^*(x;q) \ll_{k,A} \frac{x}{(\log x)^A}.
 \]
\end{cor}

\begin{proof}
Apply Lemma \ref{BV} with $A$ replaced by $2A+k^2$, Cauchy's inequality, the trivial bound $|E^*(x;q)| \ll x/q$ and
the easy bound
\begin{equation}\label{tausum}
 \sum_{q\le y} \frac{\tau^2_{k}(q)}{q} \ll_k (\log y)^{k^2},
\end{equation}
to get 
\begin{align*}
\left(\sum_{q\le x^{1/3}} \tau_k(q) E^*(x;q)\right)^2 
& \le  \left(\sum_{q\le x^{1/3}} \tau_k(q)^2 |E^*(x;q)|\right) 
\left(\sum_{q\le x^{1/3}} |E^*(x;q)|\right)\\
& \ll_{k,A}   
x\left(\sum_{q\le x^{1/3}} \frac{\tau_k(q)^2}{q}\right)\frac{x}{(\log x)^{2A+k^2}}\\
 & \ll_{k,A}   \frac{x^2}{(\log x)^{2A}},
\end{align*}
which leads to the desired conclusion. 
\end{proof}

Finally, we need a lower bound from sieve theory.

\begin{lem}\label{sievelower}
There are absolute constants $c_1>0$ and $c_2\ge 2$ so that for $y\ge c_2$, 
$y^3\le x$, and any even positive 
 integer $b$, we have
\[
\ssum{n\in(x,2x]\\bn+1\text{ prime}\\P^-(n)>y}1
 \ge \frac{c_1 bx}{\varphi(b)\log (bx)\log y} - 2\sum_{m\le y^3}
  3^{\omega(m)}E^*(2bx;bm).
 \]
\end{lem}

\begin{proof}
 We apply a standard lower bound sieve to the set
 \[
  \mathcal{A} = \left\{ \frac{\ell-1}{b} : \ell\text{ prime},~ \ell\in (bx+1,2bx],~\ell\equiv 1\pmod{b} \right\}.
 \]
With $\mathcal{A}_d$ the set of elements of $\mathcal{A}$ divisible by a squarefree integer $d$, we have
$|\mathcal{A}_d|=Xg(d)/d+r_d$, where
\[
 X = \frac{\li(2bx)-\li(bx+1)}{\varphi(b)}, \quad g(d)=\sprod{p\,|\,d \\ p\,\nmid\, b} \frac{p}{p-1}, 
 \quad |r_d|\le 2E^*(2bx;db).
\]
It follows that for $2\le v<w$,
\[
 \sum_{v\le p<w} \frac{g(p)}{p}\log p = \log\frac{w}{v} + O(1),
\]
the implied constant being absolute.  Apply \cite[Theorem 8.3]{HR} with $q=1$, $\xi=y^{3/2}$ and $z=y$, observing
that the condition $\Omega_2(1,L)$ of \cite[p. 142]{HR} holds with 
an absolute constant $L$.   
With the function $f(u)$ as defined in \cite[pp. 225--227]{HR}, we have $f(3)=\frac23e^\gamma\log2>\frac45$.
Then with $B_{19}$ the absolute constant in \cite[Theorem 8.3]{HR}, we have
\[
 f(3)-B_{19}\frac L{(\log \xi)^{1/14}} \ge \frac12
\]
for large enough $c_2$. 
We obtain the bound
\begin{align*}
 \#\{ x<n\le 2x: bn+1 \text{ prime, } & P^-(n)>y\} 
\ge \frac{X}{2} \prod_{p\le y} \(1 - \frac{g(p)}{p} \) - \sum_{m\le \xi^2}
 3^{\omega(m)}|r_m| \\
 &\ge \frac{c_1 bx}{\varphi(b)\log (bx)\log y} - 2\sum_{m\le y^3}
  3^{\omega(m)}E^*(2bx;bm). 
\end{align*}
This completes the proof.
\end{proof}

%
%
\section{The set-up}
%
%

If $n=\lambda(p_1p_2\dots p_k)$, where $p_1,p_2,\dots,p_k$ are distinct primes,
then we have $n=\lcm[p_1-1,p_2-1,\dots,p_k-1]$.  If we further assume
that $n$ is squarefree and consider the Venn diagram with the sets
$S_1,\ldots,S_k$ of the prime factors of $p_1-1,\ldots,p_k-1$, respectively, 
then this equation gives an ordered
factorization of $n$ into $2^k-1$ factors (some of which may be the
trivial factor $1$).  Here we ``see" the shifted primes
$p_i-1$ as products of certain subsequences of $2^{k-1}$ of these factors.  Conversely, given $n$
and an ordered factorization of $n$ into $2^k-1$ factors, we can ask how likely it is
for those $k$ products of $2^{k-1}$ factors to all be shifted primes.  Of course, this is not
likely at all, but if $n$ has many prime factors, and so many factorizations, our odds
improve that there is at least one such ``good" factorization.  For example, when
$k=2$, we factor a squarefree number $n$ as $a_1a_2a_3$, and we ask for $a_1a_2+1=p_1$ and $a_2a_3+1=p_2$ to
both be prime.  If so, we would have $n=\lambda(p_1p_2)$.  The heuristic argument from \cite{LP}
was based on this idea.  In particular, if a squarefree $n$ is even and has
at least $\theta_k\log\log n$ odd prime
factors (where $\theta_k > k/\log(2^k-1)$ is fixed and $\theta_k\to 1/\log 2$
as $k\to\infty$)
then there are so many factorizations of $n$ into $2^k-1$ factors,
that it becomes likely that $n$ is a $\lambda$-value.  The lower bound proof from \cite{LP} concentrated
just on the case $k=2$, but here we attack the general case.  As in \cite{LP}, we let $r(n)$ be
the number of representations of $n$ as the $\lambda$ of a number with $k$ primes.  To see that $r(n)$ is
often positive, we show that it's average value is large, and that the average value of $r(n)^2$ is
not much larger.  Our conclusion will follow from Cauchy's inequality.

Let $k\ge 2$ be a fixed integer, let $x$ be sufficiently large (in terms of $k$), and put
\be\label{ky}
y = \exp\Big\{ \frac{\log x}{200k \log\log x} \Big\}, \qquad l = \flr{\frac{k}{(2^k-1)\log(2^k-1)}\log\log y}.
\ee
For $n\le x$, let $r(n)$ be the number of representations of $n$ in the 
form 
\be\label{nab}
n = \prod_{i=0}^{k-1} a_i \prod_{j=1}^{2^k-1} b_j,
\ee
where $P^+(b_j) \le y < P^-(a_i)$ for all $i$ and $j$, $2\mid b_{2^k-1}$, $\omega(b_j)=l$ for each $j$, $a_i>1$ for all $i$, 
and furthermore that $a_iB_i+1$ is prime for all $i$, where
\be\label{Bi}
B_i = \prod_{\flr{j/2^i}\textrm{ odd}} b_j.
\ee
Observe that each $B_i$ is even since it is a multiple of $b_{2^k-1}$ 
(because $\lfloor (2^k-1)/2^i\rfloor=2^{k-i}-1$ is odd), 
each $B_i$ is the product of $2^{k-1}$ of the numbers $b_j$, and that
every $b_j$ divides $B_0\cdots B_{k-1}$.  
Also, if $n$ is squarefree and $r(n)>0$,
then the primes $a_iB_i+1$ are all distinct and it follows that
\[
 n = \lambda \Bigg( \prod_{i=0}^{k-1} (a_iB_i+1) \Bigg),
\]
therefore such $n\le x$ are counted by $V_{\lambda}(x)$. 
We count how often $r(n)>0$ using Cauchy's inequality in the following standard way:
\be\label{S1S2}
 \# \{ 2^{-2k} x < n \le x : \mu^2(n)=1, r(n) > 0 \} \ge \frac{S_1^2}{S_2},
\ee
where
\[ 
  S_1=\sum_{2^{-2k}x<n\le x} \mu^2(n) r(n),
 \qquad S_2 = \sum_{2^{-2k}x<n\le x} \mu^2(n) r^2(n).
\]
Our application of Cauchy's inequality is rather sharp, as we will
show below that $r(n)$ is approximately 1 on average over the kind of
integers we are interested in, both in mean and 
in mean-square.  More precisely, in the next section, we prove 
\be\label{S1lower}
S_1 \gg \frac{x}{(\log x)^{\beta_k} (\log\log x)^{O_k(1)}},
\ee
and in the final section, we prove
\be\label{S2upper}
S_2 \ll \frac{x (\log\log x)^{O_k(1)}}{(\log x)^{\beta_k}},
\ee
where
\be\label{betak}
 \beta_k = 1 - \frac{k}{\log(2^k-1)} \(1 + \log\log(2^k-1) - \log k\).
\ee
Together, the inequalities \eqref{S1S2}, \eqref{S1lower} and \eqref{S2upper} imply that
\[
 V_\lambda(x) \gg \frac{x}{(\log x)^{\beta_k} (\log\log x)^{O_k(1)}}.
\]
We deduce the lower bound of Theorem \ref{mainthm} by noting that $\lim_{k\to \infty} \beta_k = \eta$.

Throughout, constants implied by the symbols $O$, $\ll$, $\gg$, and $\asymp$ may depend on $k$, but not on any other variable.

%
%
\section{The lower bound for $S_1$}
\label{sec:4}
%
%

For convenience, when using the sieve bound in Lemma \ref{sievelower}, we consider a 
slightly larger sum $S_1'$ than $S_1$, namely 
\[
 S_1' := \sum_{n\in\mathcal{N}} r(n),
\]
where $\mathcal{N}$ is the set of $n\in (2^{-2k} x,x]$ 
of the form $n=n_0n_1$ with $P^+(n_0)\le y < P^-(n_1)$ and $n_0$ squarefree.  That is, in
$S_1'$ we no longer require the numbers $a_0,\dots,a_{k-1}$ in \eqref{nab} to be squarefree.
The difference between $S_1$ and $S_1'$ is very small; indeed, putting
$h=2^k+k-1$, note that $r(n)\le\tau_h(n)$, so that we have by \eqref{nab} the estimate 
\be\label{eq:Esti1}
\begin{split}
 S_1'-S_1 &\le \ssum{n\le x \\ \exists p>y : p^2\,|\,n} \tau_h(n) \le \sum_{p>y} 
 \ssum{n\le x \\ p^2 \,|\, n} \tau_h(n) 
 \le\sum_{p>y}  \tau_h(p^2) \sum_{m\le x/p^2} \tau_h(m)\\
 &\le \sum_{p>y } \tau_h(p^2) \frac{x}{p^2} \sum_{m\le x} \frac{\tau_h(m)}{m}
 \ll \frac{x(\log x)^h}{y}.
\end{split}\ee
Here we have used the inequality $\tau_h(uv)\le\tau_h(u)\tau_h(v)$ as well as 
the easy bound 
\begin{equation}
\label{eq:tauh}
\sum_{m\le x} \frac{\tau_h(m)}{m}\ll (\log x)^h,
\end{equation}
which is similar to \eqref{tausum}.
By \eqref{nab}, the sum  $S_1'$ counts the number of $(2^{k-1}+k)$-tuples $(a_0,\ldots,a_{k-1},b_1,
\ldots,b_{2^k-1})$ satisfying
\be\label{nab1}
 2^{-2k}x < a_0\cdots a_{k-1} b_1 \cdots b_{2^k-1} \le x
\ee
and with $P^+(b_j) \le y < P^+(a_i)$ for every $i$ and $j$, $b_1\cdots b_{2^k-1}$ squarefree, $2\mid b_{2^k-1}$,
$\omega(b_j)=l$ for every $j$, $a_i>1$ for every $i$, and $a_iB_i+1$ prime for every $i$, where $B_i$ is defined
in \eqref{Bi}.
Fix numbers $b_1,\ldots,b_{2^k-1}$.  Then
\be\label{bprod}
 b_1 \cdots b_{2^k-1} \le y^{(2^k-1)l} \le y^{2\log\log x} = x^{1/100k}.
\ee
In the above, we used the fact that $k\le2\log(2^k-1)$. 
Fix also $A_0,\ldots,A_{k-1}$, 
each a power of 2 exceeding $x^{1/2k}$, and such that 
\be\label{Arange}
 \frac{x}{2b_1\cdots b_{2^k-1}} < A_0 \cdots A_{k-1} \le \frac{x}{b_1\cdots b_{2^k-1}}.
\ee
Then \eqref{nab1} holds whenever $A_i/2 < a_i \le A_i$ for each $i$.
By Lemma \ref{sievelower}, using the facts that $B_i/\varphi(B_i)\ge 2$ 
(because $B_i$ is even) and $A_iB_i\le x$ (a consequence of \eqref{Arange}), 
we deduce that the number of choices for each $a_i$ is at least
\[
 \frac{c_1 A_i}{\log x \log y} - 2 \sum_{m\le y^3} 3^{\omega(m)} E^*(A_i B_i; mB_i).
\]
Using the elementary inequality
\[
 \prod_{j=1}^k \max(0,x_j-y_j) \ge \prod_{j=1}^k x_j - \sum_{i=1}^k y_i \prod_{j\ne i} x_j,
\]
valid for any non-negative real numbers $x_j,y_j$, we find that the number of admissible
$k$-tuples $(a_0,\ldots,a_{k-1})$ is at least
\renewcommand{\AA}{\mathbf{A}}
\begin{align*}
  \frac{c_1^k A_0\cdots A_{k-1}}{(\log x \log y)^k} &- \frac{2c_1^{k-1}A_0\cdots A_{k-1}}
  {(\log x \log y)^{k-1}} \sum_{i=0}^{k-1} \frac{1}{A_i} \sum_{m\le y^3} 3^{\omega(m)} E^*(A_i B_i;mB_i)\\
  &=M(\AA,\bb) - R(\AA,\bb),
\end{align*}
say.  By symmetry and \eqref{Arange},
\begin{equation}
\label{eq:Esti2}
 \sum_{\AA,\bb} R(\AA,\bb) \ll \frac{x} {(\log x \log y)^{k-1}}  \sum_{\bb} 
 \frac{1}{b_1\cdots b_{2^k-1}} \sum_{\AA} \frac{1}{A_0} \sum_{m\le y^3} 3^{\omega(m)} E^*(A_0B_0;mB_0),
\end{equation}
where the sum on $\bb$ is over all $(2^{k}-1)-$tuples satisfying
$b_1\cdots b_{2^k-1} \le x^{1/100k}$.  Write $b_1\cdots b_{2^k-1}=B_0 B_0'$,
where $B_0'=b_2b_4\cdots b_{2^k-2}$.  Given $B_0$ and $B_0'$, the number
of corresponding tuples $(b_1,\ldots,b_{2^k-1})$ is at most $\tau_{2^{k-1}}(B_0) \tau_{2^{k-1}-1} (B_0')$.
Suppose $D/2<B_0\le D$, where $D$ is a power of 2.
Since $E^*(x;q)$ is an increasing function of $x$, $E^*(A_0B_0;mB_0)\le E^*(A_0D;mB_0)$.
Also, $3^{\omega(m)}\le \tau_3(m)$ and
\[
\sum_{B_0'\le x} \frac{\tau_{2^{k-1}-1} (B_0')}{B_0'} \ll (\log x)^{2^{k-1}-1}.
\]
(this is \eqref{eq:tauh} with $h$ replaced by $2^{k-1}-1$). We therefore deduce that
\[
 \sum_{\AA,\bb} R(\AA,\bb) \ll \frac{x(\log x)^{2^{k-1}-1}} {(\log x \log y)^{k-1}}
 \sum_{\AA} \frac{1}{A_0} \sum_{D} \frac{1}{D} \ssum{D/2<B_0\le D \\ m\le y^3} 
 \tau_3(m) \tau_{2^{k-1}}(B_0) E^*(A_0D;mB_0),
 \]
the sum being over $(A_0,\ldots,A_{k-1},D)$, each a power of 2, $D\le x^{1/100k}$, $A_i\ge x^{1/2k}$ for each $i$
and $A_0\cdots A_{k-1}D \le x$.  With $A_0$ and $D$ fixed, the number of choices for $(A_1,\ldots,A_{k-1})$
is $\ll (\log x)^{k-1}$.  Writing $q=mB_0$, we obtain
\begin{align*}
\sum_{\AA,\bb}&R(\AA,\bb)\\
&\ll x \frac{(\log x)^{2^{k-1}-1}}{(\log y)^{k-1}} \sum_{D\le x^{1/100k}} \sum_{x^{1/2k}<A_0 \le x/D} \frac{1}{A_0 D} 
\sum_{q\le y^3 x^{1/100k}} \tau_{2^{k-1}+3}(q) E^*(A_0 D;q)\\
 &\ll \frac{x}{(\log x)^{\beta_k+1}},
\end{align*}
where we used Corollary \ref{BVcor} in the last step with $A=2^{k-1}-k+4+\beta_k$.  

For the main term, by \eqref{Arange}, given any
$b_1,\ldots,b_{2^{k-1}}$, the product $A_0\cdots A_{k-1}$ is determined (and larger than $\frac12 x^{1-1/100k}$
by \eqref{bprod}),
so there are
$\gg (\log x)^{k-1}$ choices for the $k$-tuple $A_0,\ldots,A_{k-1}$.  Hence,
\[
 \sum_{\AA,\bb} M(\AA,\bb) \gg \frac{x}{(\log y)^k \log x} \sum_\bb \frac{1}{b_1\cdots b_{2^k-1}}.
\]
Let $b=b_1\cdots b_{2^k-1}$.  Given an even, squarefree integer $b$, 
the number of ordered factorizations of $b$ as $b=b_1\cdots b_{2^k-1}$,
where each $\omega(b_i)=l$ and $b_{2^k-1}$ is even,
is equal to
${\displaystyle{ \frac{((2^k-1)l)!}{(2^k-1)(l!)^{2^k-1}}}}$. Let $b'=b/2$, so ${\displaystyle{ h:= \omega(b')=(2^k-1)l-1 = \frac{k\log\log y}{\log(2^k-1)} + O(1)}}$. 
Applying Lemma \ref{omega}, Stirling's formula and the fact that $(2^k-1)l=h+O(1)$, produces
\begin{align*}
 \sum_{\bb} \frac{1}{b_1\ldots b_{2^k-1}} &\ge \frac{((2^k-1)l)!}{2(2^k-1)(l!)^{2^k-1}} 
 \ssum{P^+(b')\le y \\ \omega(b')=h} \frac{\mu^2(b')}{b'} \\
 &\gg \frac{((2^k-1)l)!}{(l!)^{2^k-1}} \frac{(\log\log y)^h}{h!}
=\frac{(\log\log y)^h}{(l!)^{2^k-1}}(\log\log x)^{O(1)}\\
 &= \left[ \frac{(2^k-1) \er \log (2^k-1)}{k} \right]^{(2^k-1)l} (\log\log x)^{O(1)}\\ 
 & =  (\log y)^{\frac{k}{\log(2^k-1)} \log\left[\frac{(2^k-1)\er \log(2^k-1)}{k}\right]} (\log\log x)^{O(1)}\\ 
 & =  (\log y)^{k -\beta_k+1} (\log\log x)^{O(1)}.
\end{align*}
Invoking \eqref{ky}, we obtain that
\begin{equation}
\label{eq:Esti3}
 \sum_{\AA,\bb} M(\AA,\bb) \ge \frac{x}{(\log x)^{\beta_k} (\log\log x)^{O(1)}}.
\end{equation}
Inequality \eqref{S1lower} now follows from the above estimate \eqref{eq:Esti3} and 
our earlier estimates \eqref{eq:Esti1} of $S_1'-S_1$ and \eqref{eq:Esti2} of $\sum_{\AA,\bb} R(\AA,\bb)$.

%
%
\section{A multivariable sieve upper bound}
%
%

Here we prove an estimate from sieve theory that will be useful in
our treatment of the upper bound for $S_2$.

\begin{lem}\label{sieveupper}
Suppose that
\begin{itemize}
\item $y,x_1,\dots,x_h$ are reals with $3<y\le 2 \min\{x_1,\ldots,x_h\}$;
\item  $I_1,\ldots,I_k$ are nonempty subsets of $\{1,\ldots,h\}$;
\item $b_1,\ldots,b_k$ are positive integers such that if $I_i=I_j$, then $b_i\ne b_j$.  
\end{itemize}
For ${\rm\bf n}=(n_1,\dots,n_h)$, a vector of positive integers
and for $1\le j\le k$, let
$N_j=N_j({\rm\bf n})=\prod_{i\in I_j} n_i$.  Then
\begin{multline*}
 \# \{ {\rm\bf n}: x_i<n_i\le 2x_i\, (1\le i\le h),\, P^-(n_1\cdots n_h)>y,
\, b_j N_j+1\text{ prime } (1\le j\le k) \} \\
 \ll_{h,k} \frac{x_1\cdots x_h}{(\log y)^{h+k}} \big(\log\log (3b_1\cdots b_k) \big)^{k}.
\end{multline*}
\end{lem}

\begin{proof}
Throughout this proof, all Vinogradov symbols $\ll$ and $\gg$ as well as the Landau symbol $O$ depend on both 
$h$ and $k$. Without loss of generality, suppose that $y \le (\min(x_i))^{1/(h+k+10)}$.  
Since $n_i>x_i\ge y^{h+k+10}$ for every $i$, we see that the number of $h$-tuples in question does not exceed
\[
 S:=\# \{\textbf{n} : x_i<n_i\le 2x_i\, (1\le i\le h),\, P^-(n_1\cdots n_h (b_1N_1+1)\cdots(b_{k}N_{k}+1))>y\}.
\]
We estimate $S$ in the usual way with sieve methods, although this 
is a bit more general than the standard applications and we give the proof in some detail (the case $h=1$ being 
completely standard).  Let $\mathcal{A}$ denote the multiset
\[
 \mathcal{A}=\Big\{ n_1 \cdots n_h \prod_{j=1}^k (b_j N_j+1) : x_j < n_j \le 2x_j \, (1\le j\le h) \Big\}.
\]
For squarefree $d\le y^2$ composed of primes $\le y$, we have by a simple counting argument
\[
 |\mathcal{A}_d| := \# \{ a\in\mathcal{A}  : d\mid a \} = \frac{\nu(d)}{d^h}X + r_d,
\]
where $X=x_1\cdots x_h$, $\nu(d)$ is the number of solution vectors {\bf n} 
modulo $d$ of the congruence
\[
 n_1 \cdots n_h \prod_{j=1}^k (b_j N_j+1) \equiv 0\pmod{d},
\]
and the remainder term satisfies, for $d\le \min(x_1,\ldots,x_h)$,
\begin{eqnarray*}
 |r_d|  & \le &  \nu(d) \sum_{i=1}^h 
\prod_{\substack{1\le l\le h\\ l\ne i}} 
\left(\left\lfloor \frac{x_l}{d}\right\rfloor+1\right) 
\le\nu(d)\sum_{i=1}^h \frac{(x_1+d)\cdots (x_h+d)}{(x_i+d) d^{h-1}}\\
 &  \ll & \frac{\nu(d) X}{d^{h-1}\min(x_i)}.
\end{eqnarray*}
The function $\nu(d)$ is clearly multiplicative and satisfies the global upper bound $\nu(p)\le (h+k)p^{h-1}$ for 
every $p$.  If $\nu(p)=p^h$ for some $p\le y$, then clearly $S=0$.  Otherwise, the hypotheses of 
\cite[Theorem 6.2]{HR} (Selberg's sieve) are clearly satisfied, with $\kappa=h+k$, and we deduce that
\[
 S \ll X \prod_{p\le y} \( 1 - \frac{\nu(p)}{p^h} \) + \ssum{d\le y^2 \\ P^+(d)\le y} \mu^2(d) 3^{\omega(d)}|r_d|.
\]
By our initial assumption about the size of $y$,
\[
 \sum_{d\le y^2} \mu^2(d) 3^{\omega(d)}|r_d| \ll \frac{X}{\min(x_i)}\sum_{d \le y^2} (3k+3h)^{\omega(d)} \ll
 \frac{Xy^3}{\min(x_i)}\ll \frac{X}{y}.
\]
For the main term, consideration only of the congruence $n_1\cdots n_h\equiv 0\pmod{p}$ shows that
\[\nu(p) \ge h(p-1)^{h-1}=hp^{h-1}+O(p^{h-2})\] for all $p$.  On the other hand, suppose that $p\nmid b_1\cdots b_k$
and furthermore that $p\nmid (b_i-b_j)$ whenever $I_i=I_j$.  
Each congruence $b_jN_j+1\equiv 0\pmod{p}$
has $p^{h-1}+O(p^{h-2})$ solutions with $n_1\dots n_h\not\equiv0\pmod p$, 
and any two of these congruences have $O(p^{h-2})$ common solutions. Hence,
$\nu(p)=(h+k)p^{h-1}+O(p^{h-2})$.  In particular,
\begin{equation}
\label{eq:nu}
\frac{h}{p}+O\left(\frac{1}{p^2}\right)\le \frac{\nu(p)}{p^h}\le \frac{h+k}{p}+O\left(\frac{1}{p^2}\right).
\end{equation}
Further, writing $E=b_1\cdots b_k \prod_{i\ne j} |b_i-b_j|$, the upper bound \eqref{eq:nu} above is in fact an equality except when $p\mid E$.  We obtain
\[
 \prod_{p\le y} \( 1 - \frac{\nu(p)}{p^h} \) \ll \prod_{p\le y} \(1-\frac{1}{p} \)^{k+h} \prod_{p\,|\,E} \(1-\frac{1}{p} \)^{-k}
 \ll \frac{(E/\varphi(E))^{k}}{(\log y)^{h+k}} \ll \frac{(\log\log 3E)^k}{(\log y)^{h+k}}
\]
and the desired bound follows.
\end{proof}

%
%
\section{The upper bound for $S_2$}
%
%

Here $S_2$ is the number of solutions of 
\be\label{nab2}
n = \prod_{i=0}^{k-1} a_i \prod_{j=1}^{2^k-1} b_j = \prod_{i=0}^{k-1} a_i' \prod_{j=1}^{2^k-1} b_j',
\ee
with $2^{-2k} x < n\le x$, $n$ squarefree, 
\[
P^+(b_1b_1'\cdots b_{2^k-1}b_{2^k-1}') \le y < P^-(a_0a_0'\cdots a_{k-1} a_{k-1}'),
\]
$\omega(b_j)=\omega(b_j')=l$ for every $j$, $a_i>1$ for every $i$, $2\mid b_{2^k-1}$, $2\mid b'_{2^k-1}$,
and $a_iB_i+1$ and $a_i'B_i'+1$ prime for $0\le i\le k-1$,
where $B_i'$ is defined analogously to $B_i$ (see \eqref{Bi}).  Trivially, we have
\be\label{ab2}
a := \prod_{i=0}^{k-1} a_i = \prod_{i=0}^{k-1} a_i', \qquad b := \prod_{j=1}^{2^k-1} b_j = \prod_{j=1}^{2^k-1} b_j'.
\ee

We partition the solutions of \eqref{nab2} according to the number of the primes $a_iB_i+1$ that are equal to one of
the primes $a_j'B_j'+1$, a number which we denote by $m$.  
By symmetry (that is, by appropriate permutation of the vectors $(a_0,\ldots,a_{k-1})$, $(a_0',\ldots,a_{k-1})$,
$(b_1,\ldots,b_{2^k-1})$ and $(b_1',\ldots,b_{2^k-1}')$
\footnote{The permutations may be described explicitly. Suppose that
$m\le k-1$ and that
we wish to permute $(b_1,\ldots,b_{2^k-1})$ in order that
$B_{i_1},\ldots,B_{i_m}$ become $B_0,\ldots,B_{m-1}$, respectively.
Let $S_i=\{1\le j\le 2^k-1:\flr{j/2^i} \text{ odd} \}$.  The Venn
diagram for the sets $S_{i_1},\cdots,S_{i_m}$ has $2^m-1$ components of
size $2^{k-m-1}$ and one component of size $2^{k-m-1}-1$, 
and we map the variables $b_j$ with $j$ in a given component to the
variables whose indices are in
the corresponding component of the Venn diagram for
$S_0,\ldots,S_{m-1}$.}), 
without loss of generality we may suppose that $a_i B_i = a_i'B_i'$
for $0\le i\le m-1$ and that 
\be\label{nonequal}
 a_i B_i \ne a_j B_j \qquad  (i\ge m, j\ge m).
\ee
Consequently,
\be\label{aiBi}
 a_i = a_i' \; (0\le i\le m-1), \qquad B_i = B_i' \; (0\le i\le m-1).
\ee

Now fix $m$ and all the $b_j$ and $b_j'$.  For $0\le i\le m-1$, place $a_i$
into a dyadic interval $(A_i/2,A_i]$, where $A_i$ is a power of 2.  
The primality conditions on the remaining variables are now coupled with the condition
\[
 a_m \cdots a_{k-1} = a_m' \cdots a_{k-1}'.
\]
To aid the bookkeeping, let $\a_{i,j} =\gcd(a_i,a_j')$ for $m\le i,j\le k-1$.  Then
\be\label{aalpha}
 a_i = \prod_{j=m}^{k-1} \a_{i,j}, \qquad a_j' = \prod_{i=m}^{k-1} \a_{i,j}.
\ee
As each $a_i>1,a_j'>1$, each product above contains at least one factor that is greater than 1. 
Let $I$ denote the set of pairs of indices $(i,j)$ such that
$\a_{i,j}>1$, and fix  $I$.  For $(i,j)\in I$, place $\a_{i,j}$ into a dyadic interval $(A_{i,j}/2,A_{i,j}]$, where
$A_{i,j}$ is a power of 2 and $A_{i,j} \ge y$.  By the assumption on the range of $n$, we have
\be\label{Aprod}
 A_0 \cdots A_{m-1} \prod_{(i,j)\in I} A_{i,j} \order \frac{x}{b}.
\ee
For $0\le i\le m-1$, we use Lemma \ref{sieveupper} (with $h=1$) to deduce that the number of
 $a_i$ with $A_i/2<a_i\le A_i$, $P^-(a_i)>y$ and $a_iB_i+1$ prime is
 \be\label{S2ai}
  \ll \frac{A_i \log\log B_i}{\log^2 y} \ll \frac{A_i (\log\log x)^3}{\log^2 x}.
 \ee
 Counting the vectors $(\a_{i,j})_{(i,j)\in I}$ subject to the conditions:
 \begin{itemize}
 \item $A_{i,j}/2<\a_{i,j}\le A_{i,j}$ and $P^-(\a_{i,j})>y$ for $(i,j)\in I$;
 \item $a_iB_i+1$ prime $(m\le i\le k-1)$;
 \item $a_j'B_j'+1$ prime $(m\le j\le k-1)$;
 \item condition \eqref{aalpha}
 \end{itemize}
 is also accomplished with Lemma \ref{sieveupper}, this time with $h=|I|$ and with 
 $2(k-m)$ primality conditions.  The hypothesis in the lemma concerning identical sets $I_i$, which may occur
 if $\a_{i,j}=a_i=a_j'$ for some $i$ and $j$, is satisfied by our assumption \eqref{nonequal}, which implies
 in this case that $B_i\ne B_j'$.
 The number of such vectors is at most
 \be\label{S2alphaij}
  \ll \frac{\prod_{(i,j)\in I} A_{i,j} (\log\log x)^{2k-2m}}{(\log y)^{|I|+2k-2m}} \ll
  \frac{\prod_{(i,j)\in I} A_{i,j} (\log\log x)^{|I|+4k-4m}}{(\log x)^{|I|+2k-2m}}.
 \ee
Combining the bounds \eqref{S2ai} and \eqref{S2alphaij}, and recalling \eqref{Aprod}, 
we see that the number of possibilities for the $2k$-tuple 
$(a_0,\ldots,a_{k-1},a_0'\ldots,a_{k-1}')$ is at most
\[
 \ll \frac{x (\log\log x)^{O(1)}}{b (\log x)^{|I|+2k}}.
\]
With $I$ fixed, there are $O((\log x)^{|I|+m-1})$ choices for the numbers $A_0,\ldots,A_{m-1}$ and the numbers
$A_{i,j}$ subject to \eqref{Aprod}, and there are $O(1)$ possibilities for $I$.  We infer that with 
$m$ and all of the $b_j,b'_j$ fixed, the number of possible $(a_0,\ldots,a_{k-1},a_0'\ldots,a_{k-1}')$ is bounded by
\[
 \ll \frac{x (\log\log x)^{O(1)}}{b(\log x)^{2k+1-m}}.
\]

We next prove that the identities in \eqref{aiBi} imply that
\be\label{Bv}
 B_{\vv} = B_{\vv}' \qquad (\vv\in \{0,1\}^m),
\ee
where $B_\vv$ is the product of all $b_j$ where the $m$ least significant
base-2 digits of $j$ are given by the vector $\vv$,
and $B_\vv'$ is defined analogously.  
Fix $\vv=(v_0,\ldots,v_{m-1})$.
For $0\le i\le m-1$ let $C_i=B_i$ if $v_i=1$ and
$C_i=b/B_i$ if $v_i=0$, and define $C_i'$ analogously.  By \eqref{Bi}, each number $b_j$, 
where the last $m$ base-2 digits of $j$ are equal to $\vv$, 
divides every $C_i$,
and no other $b_j$ has this property.  By \eqref{aiBi}, 
$C_i=C_i'$ for each $i$ and thus
\[
 C_0 \cdots C_{m-1} = C_0' \cdots C_{m-1}'.
\]
As the numbers $b_j$ are pairwise coprime,
in the above equality the primes having exponent $m$ on the left are 
exactly those dividing $B_\vv$, and
similarly the primes on the right side having exponent $m$ are 
exactly those dividing $B'_\vv$.  
This proves \eqref{Bv}.

Say $b$ is squarefree.  We count the number
of dual factorizations of $b$ compatible with both \eqref{ab2} and \eqref{Bv}.
Each prime dividing $b$ first ``chooses" which $B_\vv=B'_\vv$ to divide.
Once this choice is made, there is the choice of which
$b_j$ to divide and also which $b'_j$.  
For the $2^m-1$ vectors $\vv \ne \textbf{0}$, 
$B_\vv=B'_\vv$ is the product of $2^{k-m}$ 
numbers $b_j$ and also the product of $2^{k-m}$ numbers $b'_j$. 
Similarly, $B_\textbf{0}$ is the product of $2^{k-m}-1$ numbers $b_j$
and $2^{k-m}-1$ numbers $b'_j$.
Thus, ignoring that $\omega(b_j)=\omega(b_j')=l$ for each $j$ and that $b_{2^k-1}$ and $b'_{2^k-1}$ are
even, the number of dual factorizations of $b$ is at most
\be
\label{numb}
\((2^m-1)(2^{k-m})^2+(2^{k-m}-1)^2\)^{\omega(b)}
=\(2^{2k-m}-2^{k+1-m}+1\)^{\omega(b)}.
\ee

Let again
$$
h=\omega(b)=(2^k-1)l=\frac{k}{\log(2^k-1)}\log\log y+O(1),
$$ 
as in Section \ref{sec:4}. Lemma \ref{omega} and Stirling's formula give
\[
 \ssum{P^+(b)\le y \\ \omega(b)=h} \frac{\mu^2(b)}{b} \ll \frac{(\log\log y)^h}{h!} \ll
 \pfrac{\er \log (2^k-1)}{k}^h.
\] 
Combined with our earlier bound \eqref{numb} for the number of admissible ways to dual factor each $b$,
we obtain 
\begin{equation}
\label{eq:S2}
S_2 \ll \frac{x(\log\log x)^{O(1)}}{\log x}  \pfrac{\er \log (2^k-1)}{k}^h
\sum_{m=0}^k (\log y)^{m-2k+\frac{k}{\log(2^k-1)}\log(2^{2k-m}-2^{k+1-m}+1)}.
\end{equation}
For real $t\in [0,k]$, let $f(t)=k\log(2^{2k-t}-2^{k+1-t}+1)-(2k-t)\log(2^k-1).$  We have $f(0)=f(k)=0$ and
\[
 f''(t) = \frac{k(\log 2)^2 (2^{2k}-2^{k+1}) 2^{-t}}{(2^{2k-t}-2^{k+1-t}+1)^2} > 0.
\]
Hence, $f(t)<0$ for $0<t<k$. Thus, the sum on $m$ in \eqref{eq:S2} is $O(1)$, and \eqref{S2upper} follows. 

Theorem \ref{mainthm} is therefore proved.


\end{document}